\documentclass[12pt]{article}
\usepackage[utf8]{inputenc}
\usepackage[T1]{fontenc}
\usepackage{lmodern}
\usepackage{amsmath,amssymb,amsthm}
\usepackage{geometry}
\usepackage{indentfirst}
\usepackage{booktabs}
\usepackage{xcolor}
\usepackage{hyperref}
\newtheorem{theorem}{Theorem}
\newtheorem{lemma}{Lemma}
\newtheorem{proposition}{Proposition}
\newtheorem{corollary}{Corollary}
\theoremstyle{remark}

\theoremstyle{definition}

\title{Finiteness in Square Classes and Weighted Zero Density of Perfect Cuboids}
\author{Valery Asiryan\\[2mm]
\small\texttt{asiryanvalery@gmail.com}}
\date{\small December 1, 2026}

\begin{document}
\maketitle

\begin{abstract}
Let $M(N)$ be the number of similarity classes of positive rational perfect cuboids with labelled edges $a,b,c$, space diagonal $g$, and squarefree invariant $N=\operatorname{sf}(abcg/2)$. We prove that $M(N)$ is finite for every positive squarefree $N$ and that $M(N)\le C_0^{1+r_N}$ for an absolute constant $C_0>1$, where $r_N$ is the rank of $y^2=x^3-N^2x$. The geometric step excludes positive-dimensional real abelian cosets contained in the compatibility surface and meeting its positive region. Quantitative Mordell--Lang then gives the bound for $M(N)$. Combining this bound with the Paulsen--West rank obstruction and Smith's results on quadratic twists, we prove
\[
 \lim_{X\to\infty}\frac1X
 \sum_{\substack{N\le X\\N\text{ positive squarefree}}}M(N)^q=0
 \qquad\text{for every fixed }q>0.
\]
For $q=1$, the sum counts cuboid classes with their full multiplicities at each $N$. We also obtain a quantitative decay estimate. The argument is unconditional.


\end{abstract}


\section{Introduction and main results}\label{sec:intro}

A positive rational perfect cuboid consists of seven positive rational numbers satisfying
\begin{equation}\label{eq:cuboid}
 \begin{gathered}
 a^2+b^2=d^2,\qquad a^2+c^2=e^2,\qquad b^2+c^2=f^2,\\
 a^2+b^2+c^2=g^2.
 \end{gathered}
\end{equation}
We label the edges $a,b,c$ and identify two cuboids when all corresponding lengths differ by the same positive rational factor. Each class has a unique primitive integral representative with these edge labels.

For $u\in\mathbb{Q}_{>0}$, let $\operatorname{sf}(u)$ be the unique positive squarefree integer representing its class in $\mathbb{Q}^\times/\mathbb{Q}^{\times2}$. Define
\begin{equation}\label{eq:invariant}
 N=\operatorname{sf}\left(\frac{abcg}{2}\right),
 \qquad E_N:\ y^2=x^3-N^2x,
 \qquad r_N=\operatorname{rank} E_N(\mathbb{Q}).
\end{equation}
The integer $N$ is a similarity invariant: scaling the lengths by $v\in\mathbb{Q}_{>0}$ multiplies $abcg/2$ by the square $v^4$.

Let $\mathcal{S}$ be the set of positive squarefree integers, and put
\begin{equation}\label{eq:count-definition}
 M(N)=\#\{\text{similarity classes with labelled edges satisfying \eqref{eq:cuboid}
 with invariant }N\}.
\end{equation}
The count concerns similarity classes; without this identification, any one rational cuboid would give infinitely many scaled copies with the same $N$.

Paulsen and West~\cite[Proposition~3.1]{PW} express the cuboid condition by three rational points on a congruent-number curve. Their Theorem~4.2 gives the rank obstruction
\begin{equation}\label{eq:rank-obstruction}
 M(N)>0\quad\Longrightarrow\quad r_N\ge2.
\end{equation}
We use these results to control the number of cuboids at a fixed $N$ and then sum these numbers over the squarefree parameters.

\begin{theorem}[Finiteness at every fixed parameter]\label{thm:finiteness}
There is an absolute constant $C_0>1$ such that, for every $N\in\mathcal{S}$,
\begin{equation}\label{eq:rank-bound}
 M(N)\le C_0^{1+r_N}<\infty.
\end{equation}
There is no upper bound or other restriction on $r_N$ in this statement. One admissible choice is
\begin{equation}\label{eq:explicit-constant}
 C_0=108^{180\cdot4^{38}}.
\end{equation}
\end{theorem}

\begin{theorem}[Weighted zero density]\label{thm:density}
For every fixed real number $q>0$,
\begin{equation}\label{eq:limit}
\displaystyle
 \lim_{X\to\infty}\frac1X\sum_{\substack{N\in\mathcal{S}\\N\le X}}M(N)^q=0.
\end{equation}
More precisely, there is an absolute constant $\gamma>0$ such that for each fixed $q>0$ there are positive constants $K_q,D_q,X_q$ with
\begin{equation}\label{eq:quantitative}
 \sum_{\substack{N\in\mathcal{S}\\N\le X}}M(N)^q
 \le K_qX\exp\bigl(-\gamma L(X)^{1/2}+D_qL(X)^{1/4}\bigr)
 \quad(X\ge X_q),
\end{equation}
where $L(X)=\log\log\log X$.
\end{theorem}

All logarithms are natural, and expressions involving $L(X)$ are used only for sufficiently large $X$. The geometric contribution is the exclusion of positive-dimensional real abelian cosets in the compatibility surface that meet its positive region. This yields the bound in Theorem~\ref{thm:finiteness} for all ranks. The rank obstruction \eqref{eq:rank-obstruction} and Smith's estimate for the scarcity of twists of rank at least two~\cite{SmithI} already give an estimate for the number of parameters with $M(N)>0$. Theorem~\ref{thm:density} controls the multiplicities by combining Theorem~\ref{thm:finiteness} with these inputs and Smith's exponential rank moments~\cite{SmithII}. A related use of rank moments to count rational points in a twist family appears in~\cite[Theorem~3.8]{SmithII}.

To make the consequence for the original objects explicit, let
\[
 \mathcal C(X)=\{\text{cuboid similarity classes with labelled edges and }
                 \operatorname{sf}(abcg/2)\le X\}.
\]
The case $q=1$ gives
\begin{equation}\label{eq:classes-limit}
 \#\mathcal C(X)=\sum_{\substack{N\in\mathcal{S}\\N\le X}}M(N),
 \qquad \lim_{X\to\infty}\frac{\#\mathcal C(X)}{X}=0.
\end{equation}
This is density with respect to the specified squarefree invariant. It is not a count ordered by maximum edge length or projective height.


\section{The Paulsen--West model in positive coordinates}\label{sec:model}

We record the coordinate normalization needed below, using the correspondence of Paulsen--West~\cite[Proposition~3.1]{PW}. Their parameters satisfy
\begin{equation}\label{eq:PW}
 N\eta_i^2=t_i-t_i^3,\qquad 0<t_i<1,
 \qquad t_1t_2+t_1t_3+t_2t_3=1.
\end{equation}
For a cuboid with $abcg/2=Ns^2$, $s>0$, its half-angle parameters are
\begin{equation}\label{eq:half-angles}
 t_1=\frac{ag}{de+bc},\qquad
 t_2=\frac{bg}{df+ac},\qquad
 t_3=\frac{cg}{ef+ab}.
\end{equation}
The associated triangle has sides $af,be,cd$ and area $abcg/2$, so the curve parameter in \eqref{eq:PW} is exactly the invariant \eqref{eq:invariant}.

For the counting argument we use points on $E_N$ in the region $x>N$, $y>0$. An explicit map from the cuboid is
\begin{equation}\label{eq:point-map}
 \begin{aligned}
 x_1&=N\frac{de+ag}{bc},& y_1&=\frac{ag}{s}\,x_1,\\
 x_2&=N\frac{df+bg}{ac},& y_2&=\frac{bg}{s}\,x_2,\\
 x_3&=N\frac{ef+cg}{ab},& y_3&=\frac{cg}{s}\,x_3.
 \end{aligned}
\end{equation}
These are the usual congruent-number points obtained from the three right triangles
\[
 \frac1s(ag,bc,de),\qquad
 \frac1s(bg,ac,df),\qquad
 \frac1s(cg,ab,ef),
\]
each of area $N$. Indeed, for a rational right triangle $(U,V,H)$ with $UV=2N$, the formulas
\begin{equation}\label{eq:triangle-map}
 x=N\frac{H+U}{V},\qquad y=Ux
\end{equation}
give $x>N$, $y>0$, and $x^2-N^2=U^2x$, hence $y^2=x^3-N^2x$. The identity $U/(H+V)=(x-N)/(x+N)$ shows that
\begin{equation}\label{eq:mobius}
 t_i=\frac{x_i-N}{x_i+N}.
\end{equation}

In these coordinates, the compatibility equation in \eqref{eq:PW} becomes
\begin{equation}\label{eq:surface}
 \begin{aligned}
 F_N(x_1,x_2,x_3)
 ={}&x_1x_2x_3-N(x_1x_2+x_1x_3+x_2x_3)\\
    &-N^2(x_1+x_2+x_3)+N^3=0.
 \end{aligned}
\end{equation}
For example, the exact expansion is
\[
 1-\sum_{i<j}t_it_j
 =\frac{-2F_N(x_1,x_2,x_3)}{(x_1+N)(x_2+N)(x_3+N)}.
\]
Let $S_N\subset E_N^3$ be the inverse image of the multihomogeneous equation \eqref{eq:surface} under the three degree-two maps $x:E_N\to\mathbb P^1$. Define
\begin{equation}\label{eq:positive-set}
 S_N^+(\mathbb{Q})=\{(P_1,P_2,P_3)\in S_N(\mathbb{Q}):
              P_i\ne O,\ x(P_i)>N,\ y(P_i)>0\text{ for all }i\}.
\end{equation}

\begin{proposition}\label{prop:injection}
The map \eqref{eq:point-map} induces an injection from the classes counted by $M(N)$ into $S_N^+(\mathbb{Q})$. In particular,
\begin{equation}\label{eq:count-injection}
 M(N)\le\#S_N^+(\mathbb{Q}).
\end{equation}
\end{proposition}
\begin{proof}
The preceding formulas establish membership in $S_N^+(\mathbb{Q})$. Scaling all cuboid lengths by $v>0$ scales $s$ by $v^2$ and leaves the three points unchanged. Finally, the points determine the ratios of the labelled edges through
\[
 a:b:c=\frac{y_1}{x_1}:\frac{y_2}{x_2}:\frac{y_3}{x_3}.
\]
These ratios determine the similarity class. This proves injectivity.
\end{proof}

Only this injection is needed. The converse cuboid construction belongs to the established correspondence~\cite{PW} and is not reproved here.


\section{Geometry of the compatibility surface}\label{sec:geometry}

\subsection{Irreducibility and degree}

\begin{lemma}\label{lem:degree}
For every $N\in\mathcal{S}$, the surface $S_N$ is geometrically integral. If
\[
 D_i=\operatorname{pr}_i^{-1}(O),\qquad
 \Theta=D_1+D_2+D_3,\qquad \mathcal L=\mathcal O(3\Theta),
\]
then
\begin{equation}\label{eq:degree}
 [S_N]=2\Theta,\qquad \Theta^3=6,
 \qquad \deg_{\mathcal L}(S_N)=108.
\end{equation}
\end{lemma}
\begin{proof}
Over $\mathbb{C}$, the substitutions
\[
 x_i=Nu_i,\qquad y_i=N\sqrt N\,v_i
\]
identify all the surfaces with $S_1\subset E_1^3$. Write the three abscissas as $u,v,z$. The base surface in $(\mathbb P^1)^3$ is the closure of the graph
\begin{equation}\label{eq:graph}
 z=\frac{A}{D},\qquad
 A=uv+u+v-1,\quad D=uv-u-v-1.
\end{equation}
The polynomials $A,D$ are coprime, so $Dz-A$ is irreducible on the affine chart. The equation is homogenized separately in each of the three coordinates. Each coordinate has degree exactly one in the affine equation, so the restriction to any coordinate divisor at infinity is nonzero. Any additional component of the hypersurface on the boundary would have to be one of those divisors. Thus the homogenized equation defines the integral graph closure with function field $\mathbb{C}(u,v)$.

Its inverse image in $E_1^3$ is obtained generically by adjoining square roots of
\[
 f_1=u(u^2-1),\qquad f_2=v(v^2-1),\qquad
 f_3=z(z^2-1).
\]
Their square classes are independent. Along the irreducible divisor
\[
 A=(u+1)(v+1)-2=0,
\]
the functions $f_1,f_2$ have order zero and $f_3$ has order one. Any product $f_1^{\epsilon_1}f_2^{\epsilon_2}f_3^{\epsilon_3}$ which is a square, with $\epsilon_i\in\{0,1\}$, must therefore have $\epsilon_3=0$. Valuations along $u=0$ and $v=0$ then give $\epsilon_1=\epsilon_2=0$. The generic inverse image consequently has a field of functions of degree eight over $\mathbb{C}(u,v)$.

The map $E_1^3\to(\mathbb P^1)^3$ is finite flat of degree eight. Its base change to the integral base surface remains finite flat. Its coordinate algebra over each affine open of the base surface embeds into its generic fiber, which is a field by the preceding calculation. Hence the entire inverse image is integral, including its boundary.

The base equation has multidegree $(1,1,1)$, and $x^*\mathcal O_{\mathbb P^1}(1)=\mathcal O_{E_N}(2O)$. Therefore $[S_N]=2\Theta$. Since $D_i^2=0$ and $D_1D_2D_3=1$, one has $\Theta^3=6$. Finally,
\[
 \deg_{\mathcal L}(S_N)=(2\Theta)(3\Theta)^2=108.
\]
The degree-three line bundle $\mathcal O_{E_N}(3O)$ is very ample on each factor, so its external tensor product $\mathcal L$ is very ample.
\end{proof}

\subsection{The real obstruction to abelian cosets}

\begin{lemma}\label{lem:real-coset}
Let $p\in S_N(\mathbb{R})$ satisfy $p_i\ne O$ and $x(p_i)>N$ for $i=1,2,3$. There is no positive-dimensional abelian subvariety $B\subset E_N^3$, defined over $\mathbb{R}$, such that
\[
 p+B\subseteq S_N.
\]
\end{lemma}
\begin{proof}
The identity component $E_N(\mathbb{R})^0$ is a circle, with abscissas in $[N,+\infty]$. Write $T_+=(N,0)$, and define on this circle
\begin{equation}\label{eq:angle}
 \theta(P)=2\arctan\frac{x(P)-N}{x(P)+N},
 \qquad \theta(O)=\frac\pi2.
\end{equation}
This is continuous, takes values in $[0,\pi/2]$, and satisfies
\begin{equation}\label{eq:endpoints}
 \theta^{-1}(0)=\{T_+\},\qquad
 \theta^{-1}(\pi/2)=\{O\}.
\end{equation}

On $S_N\cap(E_N(\mathbb{R})^0)^3$, including the boundary, one has
\begin{equation}\label{eq:angle-sum}
 \theta(P_1)+\theta(P_2)+\theta(P_3)=\pi.
\end{equation}
Indeed, the fractional linear coordinates in \eqref{eq:mobius} extend to $t_i(O)=1$ and lie in $[0,1]$. The multihomogeneous equation remains exactly $\sum_{i<j}t_it_j=1$. For $\alpha_i=\arctan t_i\in[0,\pi/4]$,
\[
 \cos(\alpha_1+\alpha_2+\alpha_3)
 =\prod_i\cos\alpha_i\left(1-\sum_{i<j}t_it_j\right)=0.
\]
Since the sum of the $\alpha_i$ belongs to $[0,3\pi/4]$, it equals $\pi/2$, proving \eqref{eq:angle-sum}.

Suppose now that $p+B\subseteq S_N$. The compact connected real Lie group $B(\mathbb{R})^0$ is a positive-dimensional torus. Choose a circle subgroup $C\subseteq B(\mathbb{R})^0$ and consider the translated circle $p+C$. Its three coordinate projections lie in $E_N(\mathbb{R})^0$. Each projection is either constant or a finite covering of that circle, and at least one is nonconstant.

If at least one projection is constant, let its coordinate index be $i$, so its angle is identically $c=\theta(p_i)\in(0,\pi/2)$. Since $C$ is positive-dimensional, there is a nonconstant projection with index $j\ne i$. This projection covers the circle and reaches $T_+$, where its angle is zero. At that point on $p+C$, the remaining coordinate $k\in\{1,2,3\}\setminus\{i,j\}$ would be forced by \eqref{eq:angle-sum} to have angle $\pi-c-0 = \pi-c > \pi/2$, which is impossible since $\theta$ takes values in $[0,\pi/2]$ (regardless of whether the $k$-th projection is constant or nonconstant).

It remains to consider three nonconstant projections, with covering degrees $n_1,n_2,n_3\ge1$. For the $i$th projection, let $A_i\subset C$ be the inverse image of $T_+$ and $Z_i\subset C$ the inverse image of $O$. Then
\[
 |A_i|=|Z_i|=n_i.
\]
At a parameter in $A_i$, the other two angles must both equal $\pi/2$. Thus $A_i\subseteq Z_j\cap Z_k$ whenever $\{i,j,k\}=\{1,2,3\}$. The three sets $A_i$ are pairwise disjoint, since two zero angles cannot occur in \eqref{eq:angle-sum}. Consequently,
\[
 n_1\ge n_2+n_3,\qquad
 n_2\ge n_1+n_3,\qquad
 n_3\ge n_1+n_2.
\]
Their sum is incompatible with $n_1+n_2+n_3>0$. This contradiction proves the lemma.
\end{proof}

With $T_+=(N,0)$, the full surface contains the elliptic coset
\[
 \{(O,P,P+T_+):P\in E_N\}.
\]
Indeed, setting the first coordinate to $O$ reduces \eqref{eq:surface} to $x_2x_3-N(x_2+x_3)-N^2=0$, and translation by $T_+$ gives $x_3=N(x_2+N)/(x_2-N)$ on a dense open set. Hence $S_N(\mathbb Q)$ is infinite whenever $r_N>0$.


\section{Finiteness and a bound uniform in the parameter}\label{sec:finiteness}

\subsection{At most one positive rational point per geometric coset}

The standard rational torsion calculation for congruent-number curves gives
\begin{equation}\label{eq:torsion}
 E_N(\mathbb{Q})_{\rm tors}=\{O,(0,0),(N,0),(-N,0)\};
\end{equation}
see Koblitz~\cite{Koblitz}. We first note that the region $x>N$, $y>0$ contains at most one point of each rational torsion orbit. For $P=(x,y)$ in that region, the nonzero torsion translations satisfy
\begin{equation}\label{eq:torsion-maps}
 \begin{aligned}
 x(P+(0,0))&=-N^2/x<0,\\
 x(P+(-N,0))&=-N(x-N)/(x+N)<0,\\
 x(P+(N,0))&=N(x+N)/(x-N)>N,\\
 y(P+(N,0))&=-2N^2y/(x-N)^2<0.
 \end{aligned}
\end{equation}
Thus no nonzero rational torsion translation preserves this region.

\begin{lemma}\label{lem:one-point}
For any translate $C=t+B\subseteq S_N$ of an abelian subvariety of $E_N^3$ over $\overline{\mathbb{Q}}$,
\begin{equation}\label{eq:one-point}
 \#\bigl(C\cap S_N^+(\mathbb{Q})\bigr)\le1.
\end{equation}
\end{lemma}
\begin{proof}
There is nothing to prove if the intersection is empty. Otherwise choose $p$ in it, so $C=p+B$. Put $\Gamma=E_N(\mathbb{Q})^3$. Then
\[
 C\cap\Gamma=p+(B\cap\Gamma).
\]
Let $H$ be the Zariski closure of the subgroup $B\cap\Gamma$. It is an algebraic subgroup of $E_N^3$. Because its defining dense set consists of rational points, it is defined over $\mathbb{Q}$. In characteristic zero, the identity component $H^0$ is an abelian subvariety defined over $\mathbb{Q}$. If $B\cap\Gamma$ were infinite, $H^0$ would be positive-dimensional. Since $H\subseteq B$, the inclusion
\[
 p+H^0\subseteq S_N
\]
would contradict Lemma~\ref{lem:real-coset}.

Therefore $B\cap\Gamma$ is finite and consists of torsion points. Any two points of $C\cap S_N^+(\mathbb{Q})$ differ coordinatewise by rational torsion. Equations~\eqref{eq:torsion}--\eqref{eq:torsion-maps} show that they are equal.
\end{proof}

The argument includes cosets defined only over $\overline{\mathbb{Q}}$. It is not necessary to assume that the original $B$ descends to $\mathbb{R}$; the rational points produce the subgroup $H^0$ to which the real lemma applies.

\subsection{Quantitative Mordell--Lang}

We use the following result of David, Nakamaye, and Philippon. Let $g\ge2$, let $E$ be an elliptic curve over $\overline{\mathbb{Q}}$, let $V\subseteq E^g$ be a subvariety, and let $\Gamma\subseteq E^g(\overline{\mathbb{Q}})$ be a subgroup of finite rank $r$. For the polarization
\[
 \mathcal L=\bigotimes_{i=1}^g\operatorname{pr}_i^*\mathcal O_E(3O),
\]
the set $V\cap\Gamma$ is covered by translates of abelian subvarieties contained in $V$, whose number is at most
\begin{equation}\label{eq:DNP}
 \bigl(\deg_{\mathcal L}V\bigr)^{
 60(r+1)(g+1)^{4(\dim V+1)^2+2}}.
\end{equation}
This is~\cite[Theorem~1.11]{DNP}.

\begin{proof}[Proof of Theorem~\ref{thm:finiteness}]
Apply \eqref{eq:DNP} with $E=E_N$, $g=3$, $V=S_N$, and $\Gamma=E_N(\mathbb{Q})^3$. By Mordell--Weil, $\Gamma$ has finite rank $3r_N$. Lemma~\ref{lem:degree} gives $\dim S_N=2$ and $\deg_{\mathcal L}S_N=108$, so the number of covering cosets is at most
\[
 108^{60(3r_N+1)4^{38}}.
\]
Each coset contributes at most one point of $S_N^+(\mathbb{Q})$ by Lemma~\ref{lem:one-point}. Proposition~\ref{prop:injection} therefore gives
\begin{equation}\label{eq:explicit-bound}
 M(N)\le\#S_N^+(\mathbb{Q})
 \le108^{60(3r_N+1)4^{38}}
 \le\left(108^{180\cdot4^{38}}\right)^{1+r_N}.
\end{equation}
This proves the theorem for every $N$, with no rank restriction. The rank obstruction \eqref{eq:rank-obstruction} has not been used in this proof.
\end{proof}

For finiteness alone, ordinary Mordell--Lang and Lemma~\ref{lem:one-point} suffice. The quantitative form is needed to control the sum of the multiplicities $M(N)$ as $N$ varies. The constant in \eqref{eq:explicit-constant} is an admissible theoretical bound; no practical enumeration threshold is asserted.


\section{The rank estimates used in the density argument}\label{sec:statistics}

Let
\[
 \mathcal R(X)=\{N\in\mathcal{S}:N\le X,\ r_N\ge2\},
 \qquad L(X)=\log\log\log X.
\]
We need two consequences of Smith's work, with their precise parameter ranges.

\begin{proposition}[Rank statistics]\label{prop:statistics}
There exist positive constants $A_1,A_2,\gamma,\beta,\kappa$ and $X_0$ such that, for every $X\ge X_0$,
\begin{equation}\label{eq:rare}
 \#\mathcal R(X)\le A_1X\exp\bigl(-\gamma\sqrt{L(X)}\bigr),
\end{equation}
and, uniformly for real $m$ satisfying $1\le m<\kappa L(X)$,
\begin{equation}\label{eq:moments}
 \sum_{\substack{N\in\mathcal{S}\\N\le X}}e^{mr_N}
 \le A_2X e^{\beta m^2}.
\end{equation}
\end{proposition}
\begin{proof}
For \eqref{eq:rare}, apply Smith~\cite[Assumption~1.1(3), Theorem~1.2, Remark~1.3]{SmithI} to $E_1:y^2=x(x-1)(x+1)$. It has full rational two-torsion and no rational cyclic isogeny of degree four. The latter condition follows, for example, from the root-difference criterion in~\cite[Example~3.1, equation~(3.1)]{SmithII}: the three relevant products are $-1,2,2$, none a rational square. Smith's bound counts all nonzero integer twist parameters with $2^\infty$-Selmer corank at least two. Restricting to positive squarefree parameters and using
\[
 r_N\le \operatorname{corank}_{\mathbb{Z}_2}\operatorname{Sel}_{2^\infty}(E_N/\mathbb{Q})
\]
gives \eqref{eq:rare}.\footnote{In the alternating case of the proof of \cite[Theorem~2.14, Section~10.4]{SmithII}, the following coefficient bound suffices for interpolation at the nodes $\pm 2^j$. Let $R=2^h$, $P(z)=\prod_{0\le j<h}(1-z^2/2^{2j})$, and let $G$ be the unique polynomial of degree at most $2h-1$ that interpolates $F(z)=\sum_i a_i z^i$ at the nodes $\pm2^j$ for $0\le j<h$. Suppose $|F(\pm2^j)|\le\delta$ for $0\le j<h$, and put $B=\sum_i |a_i|R^i$. For a node $s$, write $P_s(z)=P(z)/(1-z/s)$. The product gives $|P_s(s)|\gg1$ uniformly in $s,h$, and both $P_s$ and $P$ are uniformly bounded on the unit circle. The Lagrange formula gives $|G|\ll h\delta$ on $|z|=1$ and $|G/P|\ll h\delta$ on $|z|=R$. On the latter circle $|P|\gg2^{h(h+1)}$. The quotient $(F-G)/P$ extends across the interpolation nodes. For $B<\infty$ it is holomorphic on $|z|<R$ and continuous on $|z|\le R$. Applying the maximum principle to this quotient and Cauchy's estimate on the unit circle gives $|a_i|\ll B2^{-h(h+1)}+h\delta$. In Smith's application $B\ll2^{h(h+1)/2}$. For any sufficiently small $\eta>0$, take $h=\lfloor(\log\log H)^{1/8-\eta/4}\rfloor$ and use $\delta\le\exp(-(\log\log H)^{1/4-\eta/4})$ from Proposition~10.3. The resulting coefficient bound is $\exp(-(\log\log H)^{1/4-\eta})$ for sufficiently large $H$, as needed in Section~10.4.} This inequality does not require finiteness of the Shafarevich--Tate group.

For \eqref{eq:moments}, use the integer-parameter bound of~\cite[Theorem~1.1]{SmithII}, as obtained from Theorem~2.6 and Example~3.9 there, with $E_1/\mathbb{Q}$ and quadratic extensions of $\mathbb{Q}$. Write $\mathcal K(H)$ for the quadratic fields of absolute discriminant at most $H$. For real $m\ge1$, apply the integer-parameter bound at $\lceil m\rceil\le 2m$. Monotonicity then gives, after adjusting the constants,
\begin{equation}\label{eq:Smith-field-moment}
 \sum_{K\in\mathcal K(H)}
      e^{m\operatorname{rank} E_1(K)}
 \le e^{C m^2}\#\mathcal K(H)
\end{equation}
for $H>C$ and $1\le m<c\log\log\log H$, with constants depending only on $E_1$.

For $N>1$, take $K_N=\mathbb{Q}(\sqrt N)$. Distinct squarefree $N$ give distinct fields, and
\[
 |\operatorname{Disc} K_N|\le4N,
 \qquad r_N\le\operatorname{rank} E_1(K_N).
\]
The rank inequality follows from the isomorphism
\[
 E_N\longrightarrow E_1,
 \qquad (x,y)\longmapsto\left(\frac{x}{N},\frac{y}{N\sqrt N}\right)
\]
over $K_N$. Since $\#\mathcal K(4X)=O(X)$, inequality~\eqref{eq:Smith-field-moment} implies \eqref{eq:moments} after adjusting the constants. The single $N=1$ term can be absorbed uniformly for $m\ge1$, since its rank is fixed and $m\le m^2$. Decreasing $\kappa$ if necessary gives the stated common range in terms of $L(X)$.
\end{proof}


\section{Proof of weighted zero density}\label{sec:density-proof}

\begin{proof}[Proof of Theorem~\ref{thm:density}]
Fix $q>0$ and put
\[
 \alpha=q\log C_0>0,
 \qquad L=L(X).
\]
By \eqref{eq:rank-obstruction}, every nonzero summand comes from $\mathcal R(X)$. By Theorem~\ref{thm:finiteness},
\[
 M(N)^q\le e^{\alpha(1+r_N)}.
\]
For any real $p>1$ such that $1\le\alpha p<\kappa L$, H\"older's inequality and Proposition~\ref{prop:statistics} give
\begin{align}
 \sum_{\substack{N\in\mathcal{S}\\N\le X}}M(N)^q
 &\le e^\alpha\bigl(\#\mathcal R(X)\bigr)^{1-1/p}
     \left(\sum_{\substack{N\in\mathcal{S}\\N\le X}}
                       e^{\alpha p r_N}\right)^{1/p}\notag\\
 &\le e^\alpha
       \left(A_1Xe^{-\gamma\sqrt L}\right)^{1-1/p}
       \left(A_2Xe^{\beta\alpha^2p^2}\right)^{1/p}\notag\\
 &\le e^\alpha A_3X
       \exp\left(-\gamma\sqrt L+
                 \frac{\gamma\sqrt L}{p}+\beta\alpha^2p\right),
 \label{eq:holder}
\end{align}
where $A_3=\max\{1,A_1,A_2\}$ is independent of $p$ and $X$. If $\mathcal R(X)$ is empty, the desired bound is immediate and the same displayed upper bound still holds.

Choose
\begin{equation}\label{eq:p-choice}
 p=\frac{\sqrt\gamma}{\alpha\sqrt\beta}\,L^{1/4}.
\end{equation}
Because $q$ is fixed, sufficiently large $X$ ensures $p>1$ and
\[
 1\le\alpha p=\sqrt{\gamma/\beta}\,L^{1/4}<\kappa L.
\]
Thus this choice lies in the uniform moment range. Substitution into \eqref{eq:holder} yields
\begin{equation}\label{eq:optimized}
 \sum_{\substack{N\in\mathcal{S}\\N\le X}}M(N)^q
 \le e^\alpha A_3X
       \exp\left(-\gamma L^{1/2}
                 +2\alpha\sqrt{\beta\gamma}\,L^{1/4}\right).
\end{equation}
This is \eqref{eq:quantitative}, with
\[
 K_q=e^\alpha A_3,
 \qquad D_q=2q(\log C_0)\sqrt{\beta\gamma}.
\]

Finally, $L(X)\to\infty$ and
\[
 -\gamma L(X)^{1/2}+D_qL(X)^{1/4}\longrightarrow-\infty.
\]
Dividing \eqref{eq:quantitative} by $X$ therefore gives the explicit limit
\begin{equation}\label{eq:explicit-zero}
 0\le\frac1X\sum_{\substack{N\in\mathcal{S}\\N\le X}}M(N)^q
 \le K_q\exp\left(-\gamma L(X)^{1/2}+D_qL(X)^{1/4}\right)
 \longrightarrow0.
\end{equation}
The sum includes all squarefree $N$ and all associated ranks.
\end{proof}

\begin{corollary}\label{cor:near-rate}
For fixed $q>0$ and $0<\varepsilon<\gamma$,
\[
 \sum_{\substack{N\in\mathcal{S}\\N\le X}}M(N)^q
 \ll_{q,\varepsilon}
 X\exp\left(- (\gamma-\varepsilon)
                    \sqrt{\log\log\log X}\right).
\]
\end{corollary}
\begin{proof}
For sufficiently large $X$, the positive term $D_qL(X)^{1/4}$ in \eqref{eq:quantitative} is at most $\varepsilon L(X)^{1/2}$.
\end{proof}


\section{Interpretation for the cuboid problem}\label{sec:interpretation}

Let $Q(X)=\#\{N\in\mathcal{S}:N\le X\}$. The elementary squarefree counting formula is
\begin{equation}\label{eq:squarefree-count}
 Q(X)=\frac6{\pi^2}X+O(\sqrt X).
\end{equation}
For completeness, the identity $\mu(n)^2=\sum_{d^2\mid n}\mu(d)$ gives
\[
 Q(X)=\sum_{d\le\sqrt X}\mu(d)\left\lfloor\frac{X}{d^2}\right\rfloor
 =X\sum_{d=1}^{\infty}\frac{\mu(d)}{d^2}+O(\sqrt X),
\]
and the infinite sum equals $1/\zeta(2)=6/\pi^2$.

\begin{corollary}\label{cor:probability}
For every fixed $q>0$,
\begin{equation}\label{eq:average-sf}
 \frac1{Q(X)}\sum_{\substack{N\in\mathcal{S}\\N\le X}}M(N)^q
 \longrightarrow0.
\end{equation}
In particular, the squarefree parameters admitting any perfect cuboid have relative density zero:
\begin{equation}\label{eq:support-density}
 \frac{\#\{N\in\mathcal{S}:N\le X,\ M(N)>0\}}{Q(X)}
 \longrightarrow0.
\end{equation}
\end{corollary}
\begin{proof}
Equation~\eqref{eq:average-sf} follows from \eqref{eq:limit} and \eqref{eq:squarefree-count}. Since $M(N)$ is a nonnegative integer, $\mathbf 1_{M(N)>0}\le M(N)^q$, and \eqref{eq:support-density} follows.
\end{proof}

Equivalently, if $N$ is chosen uniformly from the squarefree integers up to $X$, every fixed positive moment of the number of cuboid classes tends to zero. This controls the multiplicity at each parameter, not only the size of the set of admissible parameters. For $T\ge1$, the same bound gives
\begin{equation}\label{eq:multiplicity-tail}
 \#\{N\in\mathcal{S}:N\le X,\ M(N)\ge T\}
 \le\frac{K_qX}{T^q}
       \exp\left(-\gamma L(X)^{1/2}+D_qL(X)^{1/4}\right)
\end{equation}
whenever $X\ge X_q$. Here $q$ is fixed before $X$ tends to infinity.

If the edge labels are disregarded, the corresponding count is $M(N)/6$. Indeed, two equal positive rational edges would give an irrational face diagonal, so the three edges are distinct. A permutation that preserves a similarity class must have scaling factor one, since it preserves the product of the three edges. Thus each unlabelled class has six labellings, and the same finiteness and density conclusions follow for unlabelled cuboids.


\section{Conclusions}

For every positive squarefree $N$, the number of positive rational perfect cuboids up to similarity is finite, regardless of the rank of $E_N$. The proof also gives the uniform estimate $M(N)\le C_0^{1+r_N}$. Its essential geometric step is that no positive-dimensional real abelian coset contained in $S_N$ can pass through its positive region; rational differences then reduce every geometric coset to at most one positive rational point.

Combining this estimate with the established Paulsen--West rank obstruction and Smith's rank statistics proves that the total number of cuboid classes with invariant at most $X$ is $o(X)$. More strongly, every fixed positive moment satisfies \eqref{eq:limit}, with the quantitative decay in \eqref{eq:quantitative}. The cuboid correspondence, the rank obstruction, quantitative Mordell--Lang, and the distributional results are used with their original attribution.

These conclusions leave open whether any perfect cuboid exists, and whether the total number of similarity classes over all $N$ is finite. They establish finiteness of every individual squarefree class and a global weighted zero-density statement. No BSD, GRH, or finiteness assumption on Shafarevich--Tate groups is used.

\section*{AI Assistance}
OpenAI GPT-6.2 Sol assisted with the development of the argument, symbolic calculations, verification code, and the preparation of the exposition. The author is responsible for the mathematical content and the final presentation.


\end{document}